\documentclass{amsart}

\usepackage{latexsym,enumerate}
\usepackage{amsmath,amsthm,amsopn,amstext,amscd,amsfonts,amssymb}
\usepackage[ansinew]{inputenc}
\usepackage{verbatim}
\usepackage{graphicx}
\usepackage{pstricks}
\usepackage{wasysym}
\usepackage[all]{xy}
\usepackage{epsfig} 
\usepackage{graphicx,psfrag} 
\usepackage{subfigure}
\usepackage{pstricks,pst-node}

\usepackage{multicol}
\usepackage{color}
\usepackage{colortbl}

\newcommand{\RR}{\mathbb{R}}
\newcommand{\ZZ}{\mathbb{Z}}

\newtheorem{theorem}{Theorem}[section]
\newtheorem{lemma}[theorem]{Lemma}

\newtheorem{corollary}[theorem]{Corollary}

\newtheorem{remark}[theorem]{Remark}

\newcommand{\spb}[1]{\smallskip}
\newcommand{\mpb}[1]{\medskip}
\newcommand{\bpb}[1]{\bigskip}

\renewcommand{\a}{\alpha}

\newcommand{\D}{\Delta}

\begin{document}

\DeclareGraphicsExtensions{.jpg,.pdf,.mps,.png}

\title{\vspace{6cm} Upper and lower bounds for topological indices on unicyclic graphs}

\author[\'{A}lvaro Mart\'{\i}nez-P\'erez]{\'{A}lvaro Mart\'{\i}nez-P\'erez}
\address{ Facultad CC. Sociales de Talavera,
Avda. Real F\'abrica de Seda, s/n. 45600 Talavera de la Reina, Toledo, Spain}
\email{alvaro.martinezperez@uclm.es}

\author[Jos\'e M. Rodr{\'\i}guez]{Jos\'e M. Rodr{\'\i}guez$^{(1)}$}
\address{Departamento de Matem\'aticas, Universidad Carlos III de Madrid,
Avenida de la Universidad 30, 28911 Legan\'es, Madrid, Spain}
\email{jomaro@math.uc3m.es}
\thanks{$^{(1)}$ Corresponding author.}

\date{\today}

\begin{abstract}
The aim of this paper is to obtain new inequalities for a large family of topological indices restricted to unicyclic graphs
and to characterize the set of extremal unicyclic graphs with respect to them.
This family includes variable first Zagreb, variable sum exdeg, multiplicative second Zagreb and Narumi-Katayama indices.
Our main results provide upper and lower bounds for these topological indices on unicyclic graphs, fixing or not the maximum degree or the number of pendant vertices.
\end{abstract}

\maketitle{}


{\it Keywords: Variable first Zagreb index, variable sum exdeg index, multiplicative second Zagreb index, Narumi-Katayama index, unicyclic graphs.}

{\it 2010 AMS Subject Classification numbers: 05C07, 92E10.}

\section{Introduction}

A topological descriptor is a single number that represents a chemical structure in graph-theoretical terms via the
molecular graph.
They play a significant role in mathematical chemistry especially in the QSPR/QSAR investigations.
A topological descriptor is called a topological index if it correlates with a molecular property.
Topological indices are used to understand physicochemical properties of chemical compounds,
since they capture some properties of a molecule in a single number.
Hundreds of topological indices have been introduced and studied, starting with the
seminal work by Wiener \cite{Wi}.
The \emph{Wiener index} of $G$ is defined as
$$
W(G)=\sum_{\{ u,v\}\subseteq V(G)} d(u,v),
$$
where $\{ u,v\}$ runs over every pair of vertices in $G$.


Topological indices based on end-vertex degrees of edges have been
used over 40 years. Among them, several indices are recognized to be useful tools in
chemical researches.
Probably, the best know such descriptor is the Randi\'c connectivity
index ($R$) \cite{R}. 

Two of the main successors of the Randi\'c index are the first and second Zagreb indices,
denoted by $M_1$ and $M_2$, respectively, and introduced by Gutman et al. in \cite{GT} and \cite{GRT}.
They are defined as
$$
M_1(G) = \sum_{u\in V(G)} d_u^2,
\qquad
M_2(G) = \sum_{uv\in E(G)} d_u d_v ,
\qquad
$$
where $uv$ denotes the edge of the graph $G$ connecting the vertices $u$ and $v$, and
$d_u$ is the degree of the vertex $u$.
See the recent surveys on the Zagreb indices \cite{AGMM}, \cite{BDFG} and \cite{GMM}.

Along the paper, we will denote by $m$ and $n$, the cardinality of the sets
$E(G)$ and $V(G)$, respectively.

Mili\v{c}evi\'c and Nikoli\'c defined in \cite{MN} the \emph{variable first and second Zagreb indices} as
$$
M_1^{\a}(G) = \sum_{u\in V(G)} d_u^{\a},
\qquad
M_2^{\a}(G) = \sum_{uv\in E(G)} (d_u d_v)^\a ,
$$
with $\a \in \RR$.

Note that $M_1^{0}$ is $n$, $M_1^{1}$ is $2m$, $M_1^{2}$ is the first Zagreb index $M_1$, $M_1^{-1}$ is the inverse degree index $ID$ \cite{Faj},
$M_1^{3}$ is the forgotten index $F$, etc.;
also, $M_2^{0}$ is $m$, $M_2^{-1/2}$ is the usual Randi\'c index, $M_2^{1}$ is the second Zagreb index $M_2$,
$M_2^{-1}$ is the modified second Zagreb index \cite{NKMT}, etc.

The concept of variable molecular descriptors was proposed as a new way of
characterizing heteroatoms in molecules (see \cite{R2}, \cite{R3}), but also to assess the structural differences (e.g.,
the relative role of carbon atoms of acyclic and cyclic parts in alkylcycloalkanes \cite{RPL}).
The idea behind the variable molecular descriptors is that the variables are determined during the
regression so that the standard error of estimate for a particular studied property is as small as possible.

In the paper of Gutman and Tosovic \cite{Gutman8}, the correlation abilities of $20$ vertex-degree-based topological indices
occurring in the chemical literature were tested for the case of standard
heats of formation and normal boiling points of octane isomers.
It is remarkable to realize that the variable second Zagreb index $M_2^\alpha$
with exponent $\alpha = -1$ (and to a lesser extent with exponent $\alpha = -2$)
performs significantly better than the Randi\'c index ($R=M_2^{-1/2}$).

The variable second Zagreb index is used in the structure-boiling point modeling of benzenoid hydrocarbons \cite{NMTJ}.
Various properties and relations of these indices are discussed in several papers (see, e.g., \cite{AP}, \cite{LZhao}, \cite{LL}, \cite{SDGM}, \cite{ZWC}, \cite{ZZ}).

Several authors attribute the beginning of the
study of unicyclic graphs to Dantzig's book on linear programming (1963), in which unicyclic graphs (called \emph{pseudotrees} there) arise in the solution of certain network flow problems
\cite{Dantzig}.
Since then, the study of unicyclic graphs is a main topic in graph theory.
For instance, unicyclic graphs form graph-theoretic models of functions and occur in several algorithmic problems.

Although only about 1000 benzenoid hydrocarbons are known, the number of
possible benzenoid hydrocarbons is huge. For instance, the number of
possible benzenoid hydrocarbons with 35 benzene rings is $5.85 \times 10^{21}$ \cite{NGJ}.
Therefore, the modeling of their physico-chemical properties is very important in order
to predict properties of currently unknown species.
The main reason for use topological indices is to obtain prediction of some property of molecules (see, e.g., \cite{Gutman7}, \cite{Estrada3}, \cite{Gutman8}, \cite{RPL}).
Therefore, given some fixed parameters, a natural problem is to find the graphs that minimize (or maximize) the value of a topological index on the set of graphs
satisfying the restrictions given by the parameters (see, e.g., \cite{BE1}, \cite{BE2}, \cite{Cruz}, \cite{Das4}, \cite{Du2}, \cite{Du3}, \cite{Edwards}, \cite{Gutman32}).
The aim of this paper is to obtain new inequalities for a large family of topological indices restricted to unicyclic graphs, fixing or not the maximum degree or the number of pendant vertices, and to characterize the extremal unicyclic graphs with respect to them.
This family includes variable first Zagreb, variable sum exdeg, multiplicative second Zagreb and Narumi-Katayama indices.

Throughout this work, $G=(V (G),E (G))$ denotes a (non-oriented) finite connected simple (without multiple edges and loops) non-trivial ($E(G) \neq \emptyset$) graph.
Note that the connectivity of $G$ is not an important restriction, since if $G$ has connected components $G_1,G_2,\dots,G_r,$ then we have either
$I(G) = I(G_1) + I(G_2) + \cdots + I(G_r)$ or
$I(G) = I(G_1) I(G_2) \cdots I(G_r)$
for every index $I$ in this paper; furthermore, every molecular graph is connected.
If $G_1$ and $G_2$ are isomorphic graphs, we write $G_1 = G_2$.

\section{Schur convexity}

Given two $n$-tuples ${\bf{x}}=(x_1,\dots, x_n)$, ${\bf{y}}=(y_1,\dots ,y_n)$ with $x_1\geq x_2\geq \cdots \geq x_n$ and $y_1\geq y_2 \geq \cdots \geq y_n$, then ${\bf{x}}$ \emph{majorizes} ${\bf{y}}$
(and we write ${\bf{x}}\succ {\bf{y}}$ or ${\bf{y}}\prec {\bf{x}}$) if
\[\sum_{i=1}^k x_i \geq \sum_{i=1}^k y_i,\]
for $1\leq k \leq n-1$ and
\[\sum_{i=1}^n x_i = \sum_{i=1}^n y_i.\]

A function $\Phi \colon \RR^n \to \RR$ is called \emph{Schur-convex} if $\Phi({\bf{x}})\geq \Phi({\bf{y}})$ for all ${\bf{x}}\succ {\bf{y}}$.
Similarly, the function is \emph{Schur-concave} if $\Phi({\bf{x}})\leq \Phi({\bf{y}})$ for all ${\bf{x}}\succ {\bf{y}}$.
We say that $\Phi$ is \emph{strictly Schur-convex} (respectively, \emph{strictly Schur-concave}) if $\Phi({\bf{x}})> \Phi({\bf{y}})$ (respectively, $\Phi({\bf{x}})< \Phi({\bf{y}})$)
for all ${\bf{x}}\succ {\bf{y}}$ with ${\bf{x}} \neq {\bf{y}}$.

If
$$
\Phi({\bf{x}})=\sum_{i=1}^n f(x_i),
$$
where $f$ is a convex (respectively, concave)
function defined on a real interval, then $\Phi$ is Schur-convex (respectively, Schur-concave).
If $f$ is strictly convex (respectively, strictly concave), then $\Phi$ is strictly Schur-convex (respectively, strictly Schur-concave).

Thus, $$M_1^{\a}(G) = \sum_{u\in V(G)} d_u^{\a},$$
is strictly Schur-convex if $\a\in (-\infty,0)\cup (1,\infty)$ and strictly Schur-concave if $\a \in (0,1)$.

\section{Unicyclic graphs}

A \emph{unicyclic} graph is a graph containing exactly one cycle \cite[p.41]{Harary}.
If $G$ is a unicyclic graph with $n$ vertices, then $G$ has $n$ edges.

\smallskip

Given $n\geq 3$, let $S_{2n}$ be the set of $n$-tuples ${\bf{x}} = (x_1, x_2, \dots ,x_{n-1},x_n)$  with $x_i\in \ZZ^+$ such that $x_1 \geq x_2 \geq  \cdots \geq  x_{n} $ and $\sum_{i=1}^{n} x_i = 2n$.

\begin{remark} \label{r:31}
Consider any unicyclic graph $G$ with $n$ vertices $v_1,\dots, v_n,$
ordered in such a way that if
${\bf{x}}={\bf{x}}_{_{G}}=(x_1,\dots,x_n)$ is the $n$-tuple where $x_i$ is the degree of the vertex $v_i$, then $x_{i}\geq x_{i+1}$ for every $1\leq i \leq n-1$.
By handshaking Lemma, we have that ${\bf{x}}\in S_{2n}$.
\end{remark}

Given any function $f : [1,\infty) \rightarrow \RR$, let us define the index
$$
I_f(G) = \sum_{u\in V(G)} f(d_u).
$$
Besides, if $f$ takes positive values, then we can define the index
$$
II_f(G) = \prod_{u\in V(G)} f(d_u).
$$

\begin{lemma} \label{l:min_max uni}
If $G$ is a unicyclic graph with $n \ge 4$ vertices, then
\[ (2,\dots,2) \prec {\bf{x}}_{_{G}} \! \prec (n-1,2,2,1,\dots, 1).\]
\end{lemma}

\begin{proof}
First of all, note that $(2,\dots,2)$ and $(n-1,2,2,1,\dots, 1)$ belong to $S_{2n}$.

Let us consider \emph{${\bf{x}}={\bf{x}}_{_{G}}=(x_1,\dots,x_n)$}.
Since $G$ contains a cycle, we have $x_1\geq x_2\geq x_3\geq 2$.

Seeking for a contradiction assume that $\sum_{i=1}^k x_i<2k$ for some $1\leq k\leq n-1$.
Thus, $x_k=1$ and
$$
\sum_{i=1}^n x_i
< 2k+n-k
= n+k
< 2n,
$$
a contradiction.
Hence,
$$
\sum_{i=1}^k x_i
\ge 2k
= \sum_{i=1}^k 2,
$$
for every $1\leq k\leq n-1$ and
$$
(2,\dots,2) \prec {\bf{x}} .
$$

Since
$$
\sum_{i=k+1}^n x_i
\ge \sum_{i=k+1}^n 1
= n-k,
$$
for any $3\leq k\leq n-1$, we have
$$
\sum_{i=1}^k x_i
= 2n -\sum_{i=k+1}^n x_i
\leq n+k
= n - 1 +2+2 +\sum_{i=4}^k 1 ,
$$
for every $3\leq k\leq n-1$ (where, as usual, we assume the convention $\sum_{i=4}^3 1 = 0$).

If $k=2$, then we have
$$
\begin{aligned}
\sum_{i=3}^n x_i
& = x_3 + \sum_{i=4}^n x_i
\geq 2 + \sum_{i=4}^n 1 = n-1,
\\
x_1 + x_2
& = 2n - \sum_{i=3}^n x_i
\leq 2n-(n-1)
= n-1+2.
\end{aligned}
$$

If $k=1$, then we have
$$
\begin{aligned}
\sum_{i=2}^n x_i
& = x_2 +x_3+ \sum_{i=3}^n x_i
\geq 2 + 2 + \sum_{i=4}^n 1 = n+1,
\\
x_1
& = 2n - \sum_{i=2}^n x_i
\leq 2n-(n+1)
= n-1.
\end{aligned}
$$

Therefore,
$$
{\bf{x}}  \prec (n-1,2,2,1,\dots, 1).
$$
\end{proof}

\begin{remark} \label{r:x}
If $G$ is a unicyclic graph with $3$ vertices then $G=C_3$, and we have
$I_f(G) = 3f(2)$ and $II_f(G) = f(2)^3$.
So, it suffices to deal with graphs of at lest $4$ vertices.
\end{remark}


\begin{theorem} \label{t:Ifuni}
If $G$ is a unicyclic graph with $n \ge 4$ vertices and $f : [1,\infty) \rightarrow \RR$ is a convex function, then
\[
nf(2) \leq I_f(G) \leq f(n-1) + 2f(2) +(n-3)f(1),
\]
and both inequalities are attained.
\end{theorem}

\begin{theorem} \label{t:If2uni}
If $G$ is a unicyclic graph with $n \ge 4$ vertices and $f : [1,\infty) \rightarrow \RR$ is a concave function, then
\[f(n-1) + 2f(2) +(n-3)f(1)\leq I_f(G) \leq n f(2),
\]
and both inequalities are attained.
\end{theorem}

In a similar way, we obtain the following results,
since
$$
\log II_f(G) = \sum_{u\in V(G)} \log f(d_u),
$$
and the logarithm is an increasing function.

\begin{theorem} \label{t:IIfuni}
If $G$ is a unicyclic graph with $n \ge 4$ vertices and $f : [1,\infty) \rightarrow \RR^+$ is a function such that $\log f$ is convex, then
\[f(2)^n \leq II_f(G) \leq   f(n-1) f(2)^2 f(1)^{n-3}
\]
and both inequalities are attained.
\end{theorem}

\begin{theorem} \label{t:IIf2uni}
If $G$ is a unicyclic graph with $n \ge 4$ vertices and $f : [1,\infty) \rightarrow \RR^+$ is a function such that $\log f$ is concave, then
\[f(n-1) f(2)^2 f(1)^{n-3} \leq II_f(G) \leq  f(2)^n,
\]
and both inequalities are attained.
\end{theorem}

Recall that a vertex in a graph is \emph{pendant} if it has degree $1$.
An edge is \emph{pendant} if it contains a pendant vertex.

Let $U_n^3$ be the unicyclic graph obtained from the cycle $C_3$ by attaching $n - 3$ pendant edges to the same vertex on $C_3$.
Note that $U_n^3$ is the unique graph with degree sequence $(n-1,2,2,1,\dots, 1)$.

\smallskip

Since $t^{\a}$ is strictly convex if $\a\in (-\infty,0)\cup (1,\infty)$, Theorem \ref{t:Ifuni} allows to obtain the following result.

\begin{theorem} \label{t:m1uni}
If $G$ is a unicyclic graph with $n \ge 4$ vertices and $\a\in (-\infty,0)\cup (1,\infty)$, then
\[
n 2^\a \le M_1^{\a}(G)
\leq (n-1)^\a + 2^{\a+1} + n-3 .
\]
Moreover, the lower bound is attained if and only if $G$ is the cycle graph and the upper bound is attained if and only if
$G=U_n^3$.
\end{theorem}

The bounds in Theorem \ref{t:m1uni} are proved when $n\ge 7$ with a different argument in \cite{ZZ}.

\smallskip

Since $t^{\a}$ is strictly concave if $\a \in (0,1)$,  Theorem \ref{t:If2uni}
allows to obtain the following result.

\begin{theorem} \label{t:m1bisuni}
If $G$ is a unicyclic graph with $n \ge 4$ vertices and $\a\in (0,1)$, then
\[
(n-1)^\a + 2^{\a+1} + n-3 \le M_1^{\a}(G)
\leq n 2^\a .
\]
Moreover, the lower bound is attained if and only if $G=U_n^3$ and the upper bound is attained if and only if $G$ is the cycle graph.
\end{theorem}

The bounds in Theorem \ref{t:m1bisuni} are proved when $n\ge 7$ with a different argument in \cite{ZZ}.

\smallskip

Theorem \ref{t:m1uni} has the following consequences.

\begin{corollary} \label{l:m1uni}
If $G$ is a unicyclic graph with $n \ge 4$ vertices, then the following inequalities hold:
$$
\begin{aligned}
4n & \leq M_1(G)\leq n^2-n+6,
\\
8n & \leq F(G)\leq (n-1)^3+n+13,
\\
\frac{n}2 & \leq ID(G)\leq \frac1{n-1} +n-2.
\end{aligned}
$$
Moreover, each lower bound is attained if and only if $G=C_n$, and each upper bound is attained if and only if $G=U_n^3$.
\end{corollary}

%

\begin{corollary}
If $G$ is a unicyclic graph with $n$ vertices and $\a<1$, then
$$
M_1^{\a}(G) = O(n).
$$
\end{corollary}

\smallskip

In 2011, Vuki\v{c}evi\'c \cite{Vuki4} proposed the following topological index (and named it
as the \emph{variable sum exdeg index}) for predicting the octanol-water partition coefficient of certain chemical compounds
$$
SEI_a(G)
= \sum_{uv\in E(G)} \big(a^{d_u} + a^{d_v}\big)
= \sum_{u\in V(G)} d_u a^{d_u} ,
$$
where $a \neq 1$ is a positive real number.
Among the set of $102$ topological indices \cite{http} proposed by the International Academy of Mathematical Chemistry \cite{http0}
(respectively, among the discrete Adriatic indices \cite{VG}), the best topological index for predicting the octanol-water
partition coefficient of octane isomers has $0.29$ (respectively $0.36$) coefficient of determination.
The variable sum exdeg index allows to obtain the coefficient of determination $0.99$, for predicting the aforementioned property of octane isomers \cite{Vuki4}.
Therefore, it is interesting to study the mathematical properties of the variable sum exdeg index.
Vuki\v{c}evi\'c initiated the mathematical study of $SEI_a$ in \cite{Vuki5}.

\smallskip

If we define $f(t)=t a^{t}$, then $f''(t)=2 a^{t}\log a + t a^{t}(\log a)^2$.
Hence, $f$ is strictly convex on $[1,\infty)$ if either $a>1$ or $a\le e^{-2}$, and Theorem \ref{t:Ifuni} allows to obtain the following result.

\begin{theorem} \label{t:SEI1}
If $G$ is a unicyclic graph with $n \ge 4$ vertices and $a>1$ or $0<a\le e^{-2}$, then
\[
n 2a^2 \le SEI_a(G)
\leq (n-1)a^{n-1} + 4a^{2} + (n-3)a .
\]
Moreover, the lower bound is attained if and only if $G$ is the cycle graph and the upper bound is attained if and only if
$G=U_n^3$.
\end{theorem}

Theorem \ref{t:SEI1} was proved recently in \cite{DimitrovAli}.

\smallskip

The \emph{Narumi-Katayama index} is defined in \cite{NK} as
$$
NK(G) = \prod_{u\in V (G)} d_u .
$$
The \emph{multiplicative second Zagreb index} or \emph{modified Narumi-Katayama index}
$$
NK^*(G) = \prod_{uv\in E (G)} d_u d_v = \prod_{u\in V (G)} d_u^{d_u}
$$
was introduced in \cite{multZ2} and \cite{GSG}.

\smallskip

Since $t \log t$ is a strictly convex function and $\log t$ is a strictly concave function, theorems \ref{t:IIfuni} and \ref{t:IIf2uni} imply, respectively, the following results.

\begin{theorem} \label{t:p2uni}
If $G$ is a unicyclic graph with $n \ge 4$, then
\[
4^{n} \leq NK^*(G)\leq 16(n-1)^{n-1}.
\]
Moreover, the lower bound is attained if and only if $G=C_n$ and the upper bound is attained if and only if $G=U_n^3$.
\end{theorem}


%

\begin{theorem} \label{t:p1uni}
If $G$ is a unicyclic graph with $n \ge 4$, then
\[
4(n-1) \leq NK(G)\leq 2^{n}.
\]
Moreover, the lower bound is attained if and only if $G=U_n^3$ and the upper bound is attained if and only if $G=C_n$.
\end{theorem}

Theorems \ref{t:p2uni} and \ref{t:p1uni} were proved in \cite{GSG} and \cite{GG}, respectively, with different arguments.

\section{Unicyclic graphs with maximum degree $\D$}

Let $S_{2n}^\D$ be the set  of $n$-tuples ${\bf{x}}\in S_{2n}$ such that $x_1=\D$. Note that if $G$ is a unicyclic graph with $n$ vertices and maximum degree $\D$ and ${\bf{x}}_{_G}$ is its degree sequence, then ${\bf{x}}_{_G}\in S_{2n}^\D$.

\smallskip

If $G$ is a unicyclic graph with $n $ vertices and maximum degree 2, then $G$ is the cycle $C_n$.

\smallskip

\begin{lemma}\label{l:min uni D}
If $G$ is a unicyclic graph with $n \ge 4$ vertices and maximum degree $\D\geq 3$ and
${\bf{y}}=(y_1,y_2,\dots, y_n)$ is such that
\begin{itemize}
	\item $y_1=\D$,
	\item $y_{j}= 2$ for every $1<j\leq  n-\D+2$,
	\item $y_{j}=1$ for every $n-\D+2< j \leq  n,$
\end{itemize}
then
\[ \bf{y} \prec {\bf{x}}_{_{G}}.\]
\end{lemma}

\begin{proof} First, note that ${\bf{y}} \in S_{2n}^\D$. Suppose \emph{${\bf{x}}={\bf{x}}_{_{G}}=(x_1,\dots,x_n)$}.
Since $G$ contains a cycle, we have $\D=x_1\geq x_2\geq x_3\geq 2$.

Seeking for a contradiction assume that $\sum_{i=1}^k x_i<\D+2(k-1)$ for some $2\leq k\leq n-\D+2$.
Thus, $x_k=1$ and
$$
\sum_{i=1}^n x_i
< \D+2k-2+n-k
= \D+n+k-2
\leq 2n,
$$
a contradiction.
Hence,
$$
\sum_{i=1}^k x_i
\ge \D+2k-2,
$$
for every $2\leq k\leq n-\D+2$, and it is immediate to check that
$$
{\bf{y}} \prec {\bf{x}}.
$$
\end{proof}

As usual, we denote by $\lfloor t \rfloor$ the lower integer part of $t\in\RR$, i.e., the greatest integer less than or equal to $t$.

\begin{lemma}\label{l:max uni D}
Let $G$ be a unicyclic graph with $n \ge 4$ vertices and maximum degree $\D\geq 3$ and $q=\big\lfloor \frac{n}{\D-1}\big\rfloor$.

If $q=1$ or $n=2\D-2$, let $s=n-\D+1$
and then \[  {\bf{x}}_{_{G}} \! \prec {\bf{z}}=(\D,s,2,1,\dots,1) .\]

If $q\geq 2$ and $n \neq 2\D-2$, let $r=n-q(\D-1)+1$ and
${\bf{z}}=(z_1,z_2,\dots, z_n)$ be such that
\begin{itemize}
	\item $z_{j}=\D$, for every $1\leq j\leq q$,
	\item $z_{j}= r$ if $j=q+1$,
	\item $z_{j}=1$ for every $q+1< j \leq n$,
\end{itemize}
and then
\[  {\bf{x}}_{_{G}} \prec \bf{z} .\]
\end{lemma}

\begin{proof}
Suppose $q=1$ or $n=2\D-2$.
Note that $(\D,s,2,1,\dots,1),{\bf{z}}\in  S_{2n}^\D$.
Since for any $k$ with $3\leq k\leq n-1$,
$$
\sum_{i=k+1}^n x_i
\ge \sum_{i=k+1}^n 1
= n-k,
$$
we have that
$$
\sum_{i=1}^k x_i
= 2n -\sum_{i=k+1}^n x_i
\leq n+k
= \D+s+2+\sum_{i=4}^k 1,
$$
for every $3\leq k\leq n-1$.
Therefore, since $\D=x_1\geq x_2\geq x_3\geq 2$, it is readily seen that
$$
{\bf{x}}  \prec (\D,s,2,1,\dots, 1).
$$

Suppose $q\geq 2$ and $n \neq 2\D-2$.
Note that ${\bf{z}}\in S_{2n}^\D$.
Since for any $k$ with $q+1\leq k\leq n-1$,
$$
\sum_{i=k+1}^n x_i
\ge \sum_{i=k+1}^n 1
= n-k,
$$
we have that
$$
\sum_{i=1}^k x_i
= 2n -\sum_{i=k+1}^n x_i
\leq n+k
= q\D+r+\sum_{i=q+2}^k 1,
$$
for every $q+1\leq k\leq n-1$. Since $\D=x_1\geq x_i$ for every $i>1$, it is immediate to check that
$$
{\bf{x}}  \prec {\bf{z}}.
$$
\end{proof}

\smallskip

For any $n\geq 4$ and $3\leq \D\leq n-1$,
let $\mathcal{H}_n^\D$ be the set of graphs obtained from the cycle $C_{k}$ with $3 \le k \le n-\D+2$ by attaching to the same vertex of the cycle $\D-2$ path graphs with lengths $m_1,m_2,\dots,m_{\D-2}\ge 0$ satisfying $k+m_1+m_2+\dots +m_{\D-2} = n$.
Note that $G \in \mathcal{H}_n^\D$ if and only if it is a unicyclic graph with degree sequence ${\bf{y}}$.

Let $\mathcal{K}_n^\D$ be the set of unicyclic graphs with degree sequence ${\bf{z}}$.
We show now that $\mathcal{K}_n^\D \neq \emptyset$.
Let $q=\big\lfloor \frac{n}{\D-1}\big\rfloor$.
If $q=1$ or $n=2\D-2$, then let $K_n^\D$ be the graph obtained from the cycle $C_3$ by attaching $\D-2$ pendant vertices to some vertex and $n-1-\D$ to other.
Note that if $q=1$ or $n=2\D-2$, then $n-1-\D\leq \D-2$ and $\mathcal{K}_n^\D =\{K_n^\D\}$.
If $q>1$, $n \neq 2\D-2$ and $r=n-q(\D-1)+1=1$, then $q \neq 2$;
let $K_n^\D$ be the graph obtained from the cycle $C_q$ by attaching $\D-2$ pendant vertices to each vertex.
If $q>1$, $n \neq 2\D-2$ and $r=n-q(\D-1)+1\geq 2$, let $K_n^\D$ be the graph obtained from the cycle $C_{q+1}$ by attaching $\D-2$ pendant vertices to each vertex on the cycle except one and $r-2$ pendant vertices to this last vertex.
Thus, $K_n^\D \in \mathcal{K}_n^\D \neq \emptyset$ in any case.

\smallskip

Therefore, by lemmas \ref{l:min uni D} and \ref{l:max uni D}, we obtain the following.

\begin{theorem}\label{t:Ifuni_minD}
If $G$ is a unicyclic graph with $n \ge 4$ vertices, maximum degree $\D\geq 3$ and $f : [1,\infty) \rightarrow \RR$ is a convex function, then
\[
I_f(G) \geq f(\D)+(n-\D+1)f(2)+(\D-2)f(1),
\]
and the inequality is attained if and only if $G \in \mathcal{H}_n^\D$.
\end{theorem}

\begin{theorem}\label{t:Ifuni_maxD}
If $G$ is a unicyclic graph with $n \ge 4$ vertices and maximum degree $\D\geq 3$, $q=\big\lfloor \frac{n}{\D-1}\big\rfloor$ and $f : [1,\infty) \rightarrow \RR$ is a convex function,
then
\begin{itemize}
	\item if $q=1$ or $n=2\D-2$,
\[
I_f(G)\leq f(\D)+f(n-\D+1)+f(2)+(n-3)f(1),
\]
	\item if $q>1$ and $n \neq 2\D-2$,
\[
I_f(G)\leq q f(\D)+f(n-q(\D-1)+1)+(n-q-1)f(1),
\]
\end{itemize}
and the inequalities are attained if and only if $G \in \mathcal{K}_n^\D$.
\end{theorem}

\begin{theorem}\label{t:If2uni_maxD}
If $G$ is a unicyclic graph with $n \ge 4$ vertices, maximum degree $\D\geq 3$ and $f : [1,\infty) \rightarrow \RR$ is a concave function,
then
\[
I_f(G)\leq f(\D)+(n-\D+1)f(2)+(\D-2)f(1),
\]
and the inequality is attained if and only if $G \in \mathcal{H}_n^\D$.
\end{theorem}

\begin{theorem}\label{t:If2uni_minD}
If $G$ is a unicyclic graph with $n \ge 4$ vertices, maximum degree $\D\geq 3$, $q=\big\lfloor \frac{n}{\D-1}\big\rfloor$ and $f : [1,\infty) \rightarrow \RR$ is a concave function, then
\begin{itemize}
	\item if $q=1$ or $n=2\D-2$,
\[
I_f(G)\geq f(\D)+f(n-\D+1)+f(2)+(n-3)f(1),
\]
	\item if $q>1$ and $n \neq 2\D-2$,
\[
I_f(G)\geq q f(\D)+f(n-q(\D-1)+1)+(n-q-1)f(1),
\]
\end{itemize}
and the inequalities are attained if and only if $G \in \mathcal{K}_n^\D$.
\end{theorem}

\begin{theorem} \label{t:IIfuni_D_min}
If $G$ is a unicyclic graph with $n \ge 4$ vertices and maximum degree $\D\geq 3$ and $f : [1,\infty) \rightarrow \RR^+$ is a function such that $\log f$ is convex, then
\[
II_f(G) \geq f(\D) f(2)^{n-\D+1}f(1)^{\D-2},
\]
and the inequality is attained if and only if $G \in \mathcal{H}_n^\D$.
\end{theorem}

\begin{theorem} \label{t:IIfuni_D_max}
If $G$ is a unicyclic graph with $n \ge 4$ vertices and maximum degree $\D\geq 3$, $q=\big\lfloor \frac{n}{\D-1}\big\rfloor$ and $f : [1,\infty) \rightarrow \RR^+$ is a function such that $\log f$ is convex, then
\begin{itemize}
	\item if $q=1$ or $n=2\D-2$, $$II_f(G) \leq f(\D) f(n-\D+1)f(2)f(1)^{n-3},$$
	\item if $q>1$ and $n \neq 2\D-2$, $$II_f(G) \leq  f(\D)^q f(n-q(\D-1)+1) f(1)^{n-q-1},$$
\end{itemize}
and both inequalities are attained if and only if $G \in \mathcal{K}_n^\D$.
\end{theorem}

\begin{theorem} \label{t:IIf2uni_D_max}
If $G$ is a unicyclic graph with $n \ge 4$ vertices and maximum degree $\D\geq 3$ and $f : [1,\infty) \rightarrow \RR^+$ is a function such that $\log f$ is concave, then
\[
II_f(G) \leq f(\D) f(2)^{n-\D+1}f(1)^{\D-2},
\]
and the inequality is attained if and only if $G \in \mathcal{H}_n^\D$.
\end{theorem}

\begin{theorem} \label{t:IIf2uni_D_min}
If $G$ is a unicyclic graph with $n \ge 4$ vertices and maximum degree $\D\geq 3$, $q=\big\lfloor \frac{n}{\D-1}\big\rfloor$ and $f : [1,\infty) \rightarrow \RR^+$ is a function such that $\log f$ is concave, then
\begin{itemize}
	\item if $q=1$ or $n=2\D-2$, $$II_f(G) \geq f(\D) f(n-\D+1)f(2)f(1)^{n-3},$$
	\item if $q>1$ and $n \neq 2\D-2$, $$II_f(G) \geq  f(\D)^q f(n-q(\D-1)+1) f(1)^{n-q-1},$$
\end{itemize}
and both inequalities are attained if and only if $G \in \mathcal{K}_n^\D$.
\end{theorem}

\smallskip

Since $t^{\a}$ is strictly convex if $\a\in (-\infty,0)\cup (1,\infty)$ and strictly concave if $\a\in (0,1)$, theorems \ref{t:Ifuni_minD}, \ref{t:Ifuni_maxD}, \ref{t:If2uni_maxD} and \ref{t:If2uni_minD}  allow to obtain the following results.

\begin{theorem} \label{t:m1uni_D_min}
If $G$ is a unicyclic graph with $n \ge 4$ vertices and maximum degree $\D\geq 3$ and $\a\in (-\infty,0)\cup (1,\infty)$, then
\[
M_1^{\a}(G) \ge \D^\a + (n-\D+1)2^\a + \D-2 .
\]
Moreover, the lower bound is attained if and only if $G \in \mathcal{H}_n^\D$.
\end{theorem}

Note that Theorem \ref{t:m1uni_D_min} generalizes \cite[Theorem 4.1]{GDA}.

\begin{theorem} \label{t:m1uni_D_max}
If $G$ is a unicyclic graph with $n \ge 4$ vertices and maximum degree $\D\geq 3$, $q=\big\lfloor \frac{n}{\D-1}\big\rfloor$ and $\a\in (-\infty,0)\cup (1,\infty)$, then
\begin{itemize}
	\item if $q=1$ or $n=2\D-2$, $$M_1^{\a}(G)\leq \D^\a+(n-\D+1)^\a+2^\a+n-3,$$
	\item if $q>1$ and $n \neq 2\D-2$, $$M_1^{\a}(G)\leq q\D^\a+\big(n-q(\D-1)+1\big)^\a+n-q-1.$$
\end{itemize}
Moreover, the upper bound is attained if and only if $G \in \mathcal{K}_n^\D$.
\end{theorem}

\begin{theorem} \label{t:m1bisuni_D_max}
If $G$ is a unicyclic graph with $n \ge 4$ vertices and maximum degree $\D\geq 3$ and $\a\in (0,1)$, then
\[
M_1^{\a}(G) \le \D^\a + (n-\D+1)2^\a + \D-2 .
\]
Moreover, the upper bound is attained if and only if $G \in \mathcal{H}_n^\D$.
\end{theorem}

\begin{theorem} \label{t:m1bisuni_D_min}
If $G$ is a unicyclic graph with $n \ge 4$ vertices and maximum degree $\D\geq 3$, $q=\big\lfloor \frac{n}{\D-1}\big\rfloor$ and $\a\in (0,1)$, then
\begin{itemize}
	\item if $q=1$ or $n=2\D-2$, $$M_1^{\a}(G)\geq \D^\a+(n-\D+1)^\a+2^\a+n-3,$$
	\item if $q>1$ and $n \neq 2\D-2$, $$M_1^{\a}(G)\geq q\D^\a+\big(n-q(\D-1)+1\big)^\a+n-q-1.$$
\end{itemize}
Moreover, the lower bound is attained if and only if $G \in \mathcal{K}_n^\D$.
\end{theorem}

\smallskip

Theorems \ref{t:m1uni_D_min} and \ref{t:m1uni_D_max} have the following consequences.

\begin{corollary} \label{c:m1uni_D_min}
If $G$ is a unicyclic graph with $n \ge 4$ vertices and maximum degree $\D\geq 3$, then the following inequalities hold:
$$
\begin{aligned}
 & M_1(G)\geq \D^2+4n-3\D+2,
\\
 &  F(G)\geq \D^3+8n-7\D+6,
\\
 &  ID(G)\geq \frac{1}{\D}+\frac12 (n+\D-3).
\end{aligned}
$$
Moreover, each lower bound is attained if and only if $G \in \mathcal{H}_n^\D$.
\end{corollary}

\begin{corollary} \label{c:m1uni_D_max}
If $G$ is a unicyclic graph with $n \ge 4$ vertices and maximum degree $\D\geq 3$ and $q=\big\lfloor \frac{n}{\D-1}\big\rfloor$, then the following inequalities hold:
\begin{itemize}
	\item If $q=1$ or $n=2\D-2$,
$$
\begin{aligned}
 & M_1(G)\leq \D^2+(n-\D+1)^2+n+1,
\\
 &  F(G)\leq \D^3+(n-\D+1)^3+n+5,
\\
 &  ID(G)\leq \frac{1}{\D}+\frac{1}{n-\D+1}+n-\frac52.
\end{aligned}
$$
	\item If $q>1$ and $n \neq 2\D-2$,
$$
\begin{aligned}
 & M_1(G)\leq q\D^2+(n-q(\D-1)+1)^2+n-q-1,
\\
 &  F(G)\leq q\D^3+(n-q(\D-1)+1)^3+n-q-1,
\\
 &  ID(G)\leq \frac{q}{\D}+\frac{1}{n-q(\D-1)+1}+ n-q-1.
\end{aligned}
$$
\end{itemize}
Moreover, each upper bound is attained if and only if $G \in \mathcal{K}_n^\D$.
\end{corollary}

\smallskip

Since $t a^{t}$ is strictly convex on $[1,\infty)$ if $a>1$ or $a\le e^{-2}$, theorems \ref{t:Ifuni_minD} and \ref{t:Ifuni_maxD} allow to obtain the following results.

\begin{theorem} \label{t:SEI2}
If $G$ is a unicyclic graph with $n \ge 4$ vertices and maximum degree $\D\geq 3$ and $a>1$ or $0<a\le e^{-2}$, then
\[
SEI_a(G) \ge \D\, a^\D + (n-\D+1)\,2a^{2} + (\D-2)\,a .
\]
Moreover, the lower bound is attained if and only if $G \in \mathcal{H}_n^\D$.
\end{theorem}

\begin{theorem} \label{t:SEI3}
If $G$ is a unicyclic graph with $n \ge 4$ vertices and maximum degree $\D\geq 3$, $q=\big\lfloor \frac{n}{\D-1}\big\rfloor$ and
$a>1$ or $0<a\le e^{-2}$, then
\begin{itemize}
	\item if $q=1$ or $n=2\D-2$, $$SEI_a(G) \leq \D\, a^\D + (n-\D+1)\,a^{n-\D+1} + 2 a^2 + (n-3)\,a,$$
	\item if $q>1$ and $n \neq 2\D-2$, $$SEI_a(G) \leq q\,\D\, a^\D +\big(n-q(\D-1)+1\big)\,a^{n-q(\D-1)+1}+(n-q-1)\,a.$$
\end{itemize}
Moreover, the upper bound is attained if and only if $G \in \mathcal{K}_n^\D$.
\end{theorem}

\smallskip

Since $t \log t$ is a strictly convex function and $\log t$ is a strictly concave function, theorems \ref{t:IIfuni_D_min}, \ref{t:IIfuni_D_max}, \ref{t:IIf2uni_D_max} and \ref{t:IIf2uni_D_min} imply the following results.

\begin{theorem} \label{t:p2uni_D_min}
If $G$ is a unicyclic graph with $n \ge 4$ and maximum degree $\D\geq 3$, then
\[
NK^*(G)\geq \D^\D4^{n-\D+1}.
\]
Moreover, the lower bound is attained if and only if $G \in \mathcal{H}_n^\D$.
\end{theorem}

\begin{theorem} \label{t:p2uni_D_max}
If $G$ is a unicyclic graph with $n \ge 4$ and maximum degree $\D\geq 3$ and $q=\big\lfloor \frac{n}{\D-1}\big\rfloor$, then
\begin{itemize}
	\item if $q=1$ or $n=2\D-2$, $$NK^*(G)\leq 4\D^\D(n-\D+1)^{n-\D+1},$$
	\item if $q>1$ and $n \neq 2\D-2$, $$NK^*(G)\leq \D^{q\D} (n-q(\D-1)+1)^{n-q(\D-1)+1}.$$
\end{itemize}
Moreover, the upper bound is attained if and only if $G \in \mathcal{K}_n^\D$.
\end{theorem}

\begin{theorem} \label{t:p1uni_D_max}
If $G$ is a unicyclic graph with $n \ge 4$ and maximum degree $\D\geq 3$, then
\[
NK(G)\leq \D2^{n-\D+1}.
\]
Moreover, the upper bound is attained if and only if $G \in \mathcal{H}_n^\D$.
\end{theorem}

\begin{theorem} \label{t:p1uni_D_min}
If $G$ is a unicyclic graph with $n \ge 4$ and maximum degree $\D\geq 3$, and $q=\big\lfloor \frac{n}{\D-1}\big\rfloor$, then
\begin{itemize}
	\item if $q=1$ or $n=2\D-2$, $$NK(G)\geq 2\D(n-\D+1),$$
	\item if $q>1$ and $n \neq 2\D-2$, $$NK(G)\geq \D^q(n-q(\D-1)+1).$$
\end{itemize}
Moreover, the lower bound is attained if and only if $G \in \mathcal{K}_n^\D$.
\end{theorem}

\section{Unicyclic graphs with $p$ pendant vertices}

Given $n\geq 3$, let $S_{2n,p}$ be the set of $n$-tuples  ${\bf{x}}\in S_{2n}$ such that $x_j=1$ if and only if $j>n-p$.  Note that if $G$ is a unicyclic graph with $n$ vertices and $p$ pendant vertices and ${\bf{x}}_{_G}$ is its degree sequence, then ${\bf{x}}_{_G}\in S_{2n,p}$.

\smallskip

\begin{lemma}\label{l:min_max uni_p}
If $G$ is a unicyclic graph with $n \ge 4$ vertices and $1\leq p\leq n-3$ pendant vertices,
$m=\big\lfloor \frac{2n-p}{n-p}\big\rfloor$, $t=2n-p-m(n-p)$,
${\bf{a}}=(a_1,a_2,\dots, a_n)$ is such that
\begin{itemize}
	\item $a_j=m+1$ for every $1\leq j \leq t$,
	\item $a_{j}= m$ for every $t<j \leq  n-p$,
	\item $a_{j}=1$ for every $n-p< j \leq  n,$
\end{itemize}
and ${\bf{b}}=(b_1,b_2,\dots, b_n)$ is such that
\begin{itemize}
	\item $b_1=p+2$,
	\item $b_{j}= 2$ for every $1<j \leq  n-p$,
	\item $b_{j}=1$ for every $n-p< j \leq  n,$
\end{itemize}
then
\[ {\bf{a}} \prec {\bf{x}}_{_{G}}  \prec {\bf{b}}.\]
\end{lemma}

\begin{proof} First, note that ${\bf{a}},{\bf{b}}\in  S_{2n,p}$, $m \ge 2$ and $0 \le t \le n-p$.

Suppose ${\bf{x}}={\bf{x}}_{_{G}}=(x_1,\dots,x_n)\in S_{2n,p}$.  Seeking for a contradiction
assume that  $$\sum_{i=1}^{k}x_i< k(m+1)$$ for some $k\leq t$. Then, $x_k\leq m$ and
$$\sum_{i=1}^{n}x_i< k(m+1) +(n-p-k)m+p\leq 2n,$$
a contradiction. Therefore,
$$\sum_{i=1}^{k}x_i\geq k(m+1)$$ for every $k\leq t$.

Now assume that
$$\sum_{i=1}^{k}x_i< t(m+1)+(k-t)m=km+t$$
for some $t<k\leq n-p$. Then, $x_k\leq m$ and
$$\sum_{i=1}^{n}x_i< km+t +(n-p-k)m+p=2n$$
a contradiction. Therefore,
$$\sum_{i=1}^{k}x_i\geq t(m+1)+(k-t)m$$
for every $t<k\leq n-p$ and
$$\sum_{i=1}^{k}x_i\geq \sum_{i=1}^{k}a_i$$
for every $1 \le k < n$. Thus,
\[{\bf{a}} \prec {\bf{x}}_{_{G}}. \].

Since $x_2\geq \cdots \ge x_{n-p}=2$ and $x_j=1$ if $j>n-p$,
\[\sum_{i=k+1}^{n}x_i\geq \sum_{i=k+1}^{n}b_i\]
for every $1 \le k < n$.
Thus,
\[\sum_{i=1}^{k}x_i=2n-\sum_{i=k+1}^{n}x_i\leq 2n-\sum_{i=k+1}^{n}b_i =\sum_{i=1}^{k}b_i\]
for every $1 \le k < n$ and
\[{\bf{x}}_{_{G}}  \prec {\bf{b}}. \]
\end{proof}

Lemma \ref{l:min_max uni_p} has the following consequences.

\begin{theorem} \label{t:Ifuni_p}
If $G$ is a unicyclic graph with $n \ge 4$ vertices and $1\leq p\leq n-3$ pendant vertices, $m=\big\lfloor \frac{2n-p}{n-p}\big\rfloor$, $t=2n-p-m(n-p)$ and  $f : [2,\infty) \rightarrow \RR$ is a convex function, then
\[tf(m+1)+(n-p-t)f(m)+pf(1) \leq I_f(G) \leq f(p+2) + (n-p-1)f(2) +pf(1),
\]
and both inequalities are attained.
\end{theorem}

\begin{proof}
Since
$$
I_f(G)
= pf(1) + \sum_{u\in V(G), d_u \ge 2} f(d_u) ,
$$
and $f$ is a convex function on $[2,\infty)$,
Lemma \ref{l:min_max uni_p} gives the inequalities.
\end{proof}

\begin{theorem} \label{t:If2uni_p}
If $G$ is a unicyclic graph with $n \ge 4$ vertices and $1\leq p\leq n-3$ pendant vertices, $m=\big\lfloor \frac{2n-p}{n-p}\big\rfloor$, $t=2n-p-m(n-p)$ and $f : [2,\infty) \rightarrow \RR$ is a concave function, then
\[f(p+2) + (n-p-1)f(2) +pf(1) \leq I_f(G) \leq tf(m+1)+(n-p-t)f(m)+pf(1) ,
\]
and both inequalities are attained.
\end{theorem}

In a similar way, we obtain the following results.

\begin{theorem} \label{t:IIfuni_p}
If $G$ is a unicyclic graph with $n \ge 4$ vertices and $1\leq p\leq n-3$ pendant vertices, $m=\big\lfloor \frac{2n-p}{n-p}\big\rfloor$, $t=2n-p-m(n-p)$ and $f : [2,\infty) \rightarrow \RR^+$ is a function such that $\log f$ is convex, then
\[f(m+1)^tf(m)^{n-p-t}f(1)^p \leq II_f(G) \leq   f(p+2) f(2)^{n-p-1} f(1)^{p},
\]
and both inequalities are attained.
\end{theorem}

\begin{theorem} \label{t:IIf2uni_p}
If $G$ is a unicyclic graph with $n \ge 4$ vertices and $1\leq p\leq n-3$ pendant vertices, $m=\big\lfloor \frac{2n-p}{n-p}\big\rfloor$, $t=2n-p-m(n-p)$ and $f : [2,\infty) \rightarrow \RR^+$ is a function such that $\log f$ is concave, then
\[f(p+2) f(2)^{n-p-1} f(1)^{p} \leq II_f(G) \leq  f(m+1)^tf(m)^{n-p-t}f(1)^p,
\]
and both inequalities are attained.
\end{theorem}

Therefore, if $G$ is a unicyclic graph with $n \ge 4$ vertices and $1\leq p\leq n-3$ pendant vertices, the corresponding sharp upper and lower bounds for $M_1^\a(G)$, $M_1(G)$, $F(G)$, $ID(G)$, $SEI_a(G)$, $NK^*(G)$ and $NK(G)$ can be easily computed as in the previous sections.

\smallskip

The lower bound obtained in this way for $M_1(G)$ is proved in \cite{GKJA} with different arguments.

\smallskip

We state now just the inequalities for $SEI_a(G)$, since we obtain them in this case for a larger range of values of the parameter $a$.

\begin{theorem} \label{t:SEI5}
If $G$ is a unicyclic graph with $n \ge 4$ vertices and $1\leq p\leq n-3$ pendant vertices, $a>1$ or $0<a\le e^{-1}$,
$m=\big\lfloor \frac{2n-p}{n-p}\big\rfloor$ and $t=2n-p-m(n-p)$, then
\[
t(m+1)a^{m+1}+(n-p-t)m a^{m} +p a
\leq SEI_a(G)
\leq (p+2)a^{p+2} + (n-p-1)2a^2 +pa,
\]
and both inequalities are attained.
\end{theorem}

\begin{proof}
If we define $f(t)=t a^{t}$, then $f''(t)=2 a^{t}\log a + t a^{t}(\log a)^2$.
Hence, $f$ is strictly convex on $[2,\infty)$ if either $a>1$ or $a\le e^{-1}$.
Therefore, Theorem \ref{t:Ifuni_p} gives the inequalities.
\end{proof}

\section{Acknowledgements}
The first author was partially supported by a grant from Ministerio de Ciencia, Innovaci\'on y Universidades (PGC2018-098321-B-I00), Spain,
the second author by two grants from Ministerio de Econom{\'\i}a y Competitividad, Agencia Estatal de
Investigaci\'on (AEI) and Fondo Europeo de Desarrollo Regional (FEDER) (MTM2016-78227-C2-1-P and MTM2017-90584-REDT), Spain.

\end{document}